\documentclass[12pt,a4paper]{amsart}%
\usepackage{pgf,tikz}
\usepackage{amssymb}
\usepackage{amsmath}
\usepackage{amsfonts}
\usepackage{graphicx}
\usepackage{bm}
\usepackage{pgfplots}
\usepackage{float}%
\providecommand{\U}[1]{\protect\rule{.1in}{.1in}}
\usetikzlibrary{backgrounds}
\usetikzlibrary{arrows}
\pgfplotsset{compat=1.15}
\let\oldmathbf\mathbf
\renewcommand{\mathbf}[1]{\boldsymbol{\oldmathbf{#1}}}
\newtheorem{theorem}{Theorem}
\newtheorem*{theorem*}{Theorem}

\newtheorem{corollary}[theorem]{Corollary}

\newtheorem{example}[theorem]{Example}

\newtheorem{lemma}[theorem]{Lemma}

\newtheorem{remark}[theorem]{Remark}

\allowdisplaybreaks
\definecolor{zzttqq}{rgb}{0.6,0.2,0.}
\begin{document}
\title[Irregularities of distribution]{Irregularities of distribution for bounded sets and half-spaces}
\author[L. Brandolini]{Luca Brandolini}
\address{Dipartimento di Ingegneria Gestionale, dell'Informazione e della Produzione,
Universit\`{a} degli Studi di Bergamo, Viale Marconi 5, 24044 Dalmine BG, Italy}
\email{luca.brandolini@unibg.it}
\author[L. Colzani]{Leonardo Colzani}
\address{Dipartimento di Matematica e Applicazioni, Universit\`{a} di Milano-Bicocca,
Via Cozzi 55, 20125 Milano, Italy}
\email{leonardo.colzani@unimib.it}
\author[G. Travaglini]{Giancarlo Travaglini}
\address{Dipartimento di Matematica e Applicazioni, Universit\`{a} di Milano-Bicocca,
Via Cozzi 55, 20125 Milano, Italy}
\email{giancarlo.travaglini@unimib.it}

\begin{abstract}
We prove a general result on irregularities of distribution for Borel sets
intersected with bounded measurable sets or affine half-spaces.

\end{abstract}
\maketitle

\section{Introduction}

According to the Mathematical Reviews the term \textquotedblleft
Irregularities of Distribution\textquotedblright\ had never been used in
Mathematics before the publication of Roth's seminal paper \cite{roth}.
Starting from a celebrated conjecture of van der Corput, Roth proved the
following result, which shows that no finite sequence in the unit square can
be too evenly distributed.

\begin{theorem}
[Roth]There exists a constant $c>0$ such that for every set $\mathcal{P}%
_{N}=\left\{  z_{1},z_{2},\ldots z_{N}\right\}  $ of $N$ points in the torus
$\mathbb{T}^{2}$ we have%
\[
\int_{\mathbb{T}^{2}}\left\vert \mathrm{card}\left(  \mathcal{P}_{N}\cap
I_{x}\right)  -Nx_{1}x_{2}\right\vert ^{2}\ dx_{1}dx_{2}\geq c\log\left(
N\right)  \ ,
\]
where $I_{x}=\left[  0,x_{1}\right]  \times\left[  0,x_{2}\right]  $ for every
point $x=\left(  x_{1},x_{2}\right)  \in\left[  0,1\right)  ^{2}$.
\end{theorem}

The monograph \cite{BC} is the basic reference. See also
\cite{chazelle,DT,matousek,montgomery,travaglini}.\bigskip

The rectangles $I_{x}$ in Roth's theorem can be replaced by different families
of sets (e.g. disks or intersections of $\left[  0,1\right)  ^{2}$ with
halfplanes), and the above (sharp) $\log\left(  N\right)  $ estimate may
change drastically. See e.g. \cite{beck,BT22,montgomery,schmidt}.\medskip

More generally, as pointed out by William Chen in \cite{chen}, today many of
the problems that concern Irregularities of Distribution, also called
Geometric Discrepancy, can be formulated in the following way:

Let $\mathcal{P}_{N}$ be a set of $N$ points in $\mathbb{R}^{d}$ ($d\geq2$)
and let $E\subseteq\mathbb{R}^{d}$. We want to estimate the quality of the
distribution of these points with respect to a probability measure $\mu$
supported in $E$. We consider a reasonably large family $\mathcal{R}$ of
measurable sets and, for $R\in\mathcal{R}$, we introduce the discrepancy%
\[
\mathcal{D}_{N}\left(  R\right)  =\mathrm{card}\left(  \mathcal{P}_{N}\cap
R\right)  -N\mu\left(  R\right)  \ .
\]
\bigskip

The aim of this paper is to follow the above approach, introduce a general
point of view and prove a few theorems which extend several known results.

We may choose $E$ and $\mu$ in a fairly general way: in particular
$E\subset\mathbb{R}^{d}$ can be a $k$-dimensional manifold ($k\geq2$) and
$\mu$ the associated Hausdorff measure, or $E$ can be a fractal set endowed
with a general $\alpha$-dimensional measure $\mu$ (see below). As for the
family $\mathcal{R}$ we consider affine copies of a given body $\Omega
\subset\mathbb{R}^{d}$ with possibly fractal boundary. We also obtain results
for spherical caps (that are intersections of a given manifold and all the
affine half-spaces). See Section \ref{sect esempi} below.

\medskip

Our arguments are essentially Fourier analytic and a classical lemma of
Cassels and Montgomery is a basic tool (see \cite[Chapter 6]{montgomery}, see also
\cite{BBG,BT22}).

\section{Notation and Main Results\label{Sect:notation}}

Let $\mu$ be a positive Borel measure on $\mathbb{R}^{d}$, let $E$ be its
support, and let $\mu\left(  E\right)  =1$. Let us assume the existence of
$0<\alpha\leqslant d$ and $c>0$ such that for every $d$-dimensional open ball
$B\left(  x,r\right)  $ with center $x$ and radius $r$ we have%
\begin{equation}
\mu\left(  B\left(  x,r\right)  \right)  \leqslant c\ r^{\alpha}.
\label{Ahlfors}%
\end{equation}
We recall that, by Frostman's lemma (see e.g. \cite[Chapter 2]{Mattila}), if
$E$ is a Borel set with positive $\alpha$-dimensional Hausdorff measure, then
one can always find a finite Borel measure $\mu$ supported on $E$ and
satisfying (\ref{Ahlfors}).\bigskip

In the following we will denote by $\left\vert F\right\vert $ the Lebesgue
measure of a measurable set $F\subseteq\mathbb{R}^{d}$.

Given a Borel set $\Omega\subset\mathbb{R}^{d}$, with $\left\vert
\Omega\right\vert >0$, let $\mathcal{R}_{\Omega}$ be the family of rotated,
dilated and translated copies of $\Omega$. More precisely, for given $a$ and
$b$, we set
\[
\mathcal{R}_{\Omega}=\left\{  \left(  x+\tau\sigma\Omega\right)  \cap
E:x\in\mathbb{R}^{d},\ a\leqslant\tau\leqslant b,\sigma\in SO\left(  d\right)
\right\}
\]
and we study $\mathcal{D}_{N}\left(  R\right)  $ for $R\in\mathcal{R}_{\Omega
}$. See e.g. \cite[p. 212]{DT} for a similar point of view.

We are mainly interested in the following two cases:

\begin{itemize}
\item $\Omega$ is a bounded Borel set satisfying suitable regularity
conditions related to the Minkowski content of the boundary (see Theorem
\ref{Thm 1} and Remark \ref{Remark Minkowski}),

\item $\Omega$ is a half-space (see Theorem \ref{Thm 3}).
\end{itemize}

We obtain the discrepancy of half-spaces as limit case of the discrepancy of a
family of balls of diverging radii.

\subsection{Discrepancy for bounded sets}

Our first result exhibits a lower bound for the discrepancy associated to the
family $\mathcal{R}_{\Omega}$.

\begin{theorem}
\label{Thm 1}Let $\mu$ be a positive Borel measure on $\mathbb{R}^{d}$, with
support in a bounded set $E$ and such that $\mu\left(  E\right)  =1$. We also
assume that $\mu$ satisfies (\ref{Ahlfors}) for a given $0<\alpha\leqslant d$.
Let $\Omega\subset\mathbb{R}^{d}$ be a given bounded Borel set and assume the
existence of constants $0<\beta\leq1$ and $\kappa_{1}>0$ such that for every
$h\in\mathbb{R}^{d}$ small enough we have
\begin{equation}
\left\vert \left(  h+\Omega\right)  \bigtriangleup\Omega\right\vert
\leqslant\kappa_{1}\,\left\vert h\right\vert ^{\beta}%
\label{Diff simm da sopra}%
\end{equation}
(here $A\bigtriangleup B=\left(  A\setminus B\right)  \cup\left(  B\setminus
A\right)  $ denotes the symmetric difference). Also assume there exist
$\kappa_{2}>0$, a direction $\overline{\Theta}$, and a decreasing sequence
$t_{n}\rightarrow0$, satisfying
\begin{equation}
t_{n}\leqslant\kappa_{3}t_{n+1}\label{non lacunare}%
\end{equation}
for a suitable $\kappa_{3}>0$, such that for every $n$ we have%
\begin{equation}
\kappa_{2}\,t_{n}^{\beta}\leqslant\left\vert \left(  t_{n}\overline{\Theta
}+\Omega\right)  \bigtriangleup\Omega\right\vert .\label{Diff simm da sotto}%
\end{equation}
Then there exist positive constants $a,b$ and $c$ such that for every point
distribution $\mathcal{P}_{N}=\left\{  z_{1},z_{2},\ldots z_{N}\right\}  $ we
have%
\begin{equation}
\left\{  \int_{a}^{b}\int_{SO\left(  d\right)  }\int_{\mathbb{R}^{d}%
}\left\vert \mathcal{D}_{N}\left(  x+\tau\sigma\Omega\right)  \right\vert
^{2}dxd\sigma d\tau\right\}  ^{1/2}\geqslant c\,N^{1/2-\beta/\left(
2\alpha\right)  }.\label{Discrep basso}%
\end{equation}

\end{theorem}

In the next section we shall see that the above result is sharp.

If $E$ has positive Lebesgue measure and $\mu$ is the Lebesgue measure
restricted to $E$, then we can take $\alpha=d$. In this case the previous
result can be found in \cite[Theorem 2.10]{DT}.

\begin{remark}
\label{Remark Minkowski}The value of $\beta$ in (\ref{Diff simm da sopra}) is
related to the fractal dimension of the boundary $\partial\Omega$. It is not
difficult to show that%
\[
\left(  h+\Omega\right)  \bigtriangleup\Omega\subseteq\left\{  x\in
\mathbb{R}^{d}:\operatorname*{dist}\left(  x,\partial\Omega\right)
\leqslant\left\vert h\right\vert \right\}  .
\]
Recall that $\partial\Omega$ has finite $\left(  d-\beta\right)  $-dimensional
Minkowski content if for some $c>0$ and for every $0<t<1$ one has
\begin{equation}
\left\vert \left\{  z\in\mathbb{R}^{d}:\operatorname*{dist}\left(
z,\partial\Omega\right)  \leqslant t\right\}  \right\vert \leqslant ct^{\beta}
\label{Finite minkowski content}%
\end{equation}
(see e.g. \cite{Fal}). Hence (\ref{Finite minkowski content}) implies
(\ref{Diff simm da sopra}).
\end{remark}

The above theorem has an immediate corollary.

\begin{corollary}
\label{cor3} Let $\Omega$ be as in the previous theorem. Then there exists a
constant $c>0$ such that for every point distribution $\mathcal{P}%
_{N}=\left\{  z_{1},z_{2},\ldots z_{N}\right\}  $  there exists an affine
copy\ $\Omega^{\ast}=x^{\ast}+\tau^{\ast}\sigma^{\ast}\Omega$ \ of $\Omega$
such that%
\[
\left\vert \mathcal{D}_{N}\left(  \Omega^{\ast}\right)  \right\vert \geq
c\,N^{1/2-\beta/\left(  2\alpha\right)  }\ .
\]

\end{corollary}

\subsection{Examples}

\label{sect esempi}

We now discuss a few examples of $\mu$, $E$ and $\Omega$.

\begin{example}
It is easy to see that if $\Omega$ is a convex body, then
(\ref{Diff simm da sopra}) and (\ref{Diff simm da sotto}) hold with $\beta=1$.
We will show in the Example \ref{n gamma} that for every $0<\beta<1$ there are
sets $\Omega$ satisfying (\ref{Diff simm da sopra}) and
(\ref{Diff simm da sotto}).
\end{example}

\begin{example}
As already mentioned in the Introduction, the study of the discrepancy
associated to a point distribution on a given manifold (not necessarily
smooth) embedded in $\mathbb{R}^{d}$ is one of the main motivation of this
work. Then we consider a metric space $Y$ (with distance $\operatorname*{dist}%
$) embedded in $\mathbb{R}^{d}$ by a mapping $\Phi:Y\rightarrow\mathbb{R}^{d}$
such that%
\[
C_{1}\operatorname*{dist}\left(  y_{1},y_{2}\right)  \leqslant\left\vert
\Phi\left(  y_{1}\right)  -\Phi\left(  y_{2}\right)  \right\vert .
\]
Let $\nu$ be a $\alpha$-dimensional probability measure on $Y$, that is,%
\[
\nu\left(  B_{r}\left(  y\right)  \right)  \leqslant cr^{\alpha},
\]
where $B_{r}\left(  y\right)  $ denotes the ball centered at $y$ with radius
$r$ in the metric space $Y$. For every Borel set $F\subseteq\mathbb{R}^{d}$
define%
\[
\mu\left(  F\right)  =\nu\left(  \Phi^{-1}\left(  F\right)  \right)  .
\]
Let $x_{0}\in\mathbb{R}^{d}$. If $x_{0}\in\Phi\left(  Y\right)  $, then%
\[
\Phi^{-1}\left(  B_{r}\left(  x_{0}\right)  \right)  \subseteq B_{r/C_{1}%
}\left(  \Phi^{-1}\left(  x_{0}\right)  \right)
\]
and therefore%
\[
\mu\left(  B_{r}\left(  x_{0}\right)  \right)  \leqslant\nu\left(  B_{r/C_{1}%
}\left(  \Phi^{-1}\left(  x_{0}\right)  \right)  \right)  \leqslant c\left(
r/C_{1}\right)  ^{\alpha}\leqslant c_{2}r^{\alpha}.
\]
If $x_{0}\notin\Phi\left(  Y\right)  $ but $B_{r}\left(  x_{0}\right)
\cap\Phi\left(  Y\right)  \neq\emptyset$, then there exists $z\in\Phi\left(
Y\right)  $ such that $\left\vert x_{0}-z\right\vert <r$ and therefore%
\[
\mu\left(  B_{r}\left(  x_{0}\right)  \right)  \leqslant\mu\left(
B_{2r}\left(  z\right)  \right)  \leqslant\nu\left(  B_{2r/C_{1}}\left(
\Phi^{-1}\left(  z\right)  \right)  \right)  \leqslant c_{2}r^{\alpha}.
\]
Finally if $B_{r}\left(  x_{0}\right)  \cap\Phi\left(  Y\right)  =\emptyset$
then $\mu\left(  B_{r}\left(  x_{0}\right)  \right)  =0$. Hence $\mu$ is a
$\alpha$-dimensional measure in $\mathbb{R}^{d}$. If $\Omega$ is a given
convex body, then Theorem \ref{Thm 1} gives the estimate%
\begin{equation}
\left\{  \int_{a}^{b}\int_{SO\left(  d\right)  }\int_{\mathbb{R}^{d}%
}\left\vert \mathcal{D}_{N}\left(  x+\tau\sigma\Omega\right)  \right\vert
^{2}dxd\sigma d\tau\right\}  ^{1/2}\geqslant c\,N^{1/2-1/\left(
2\alpha\right)  }. \label{Discrepanza}%
\end{equation}

\end{example}

As a particular case of the previous example we have the following result.

\begin{corollary}
Let $\mu$ be the surface measure on a regular $k$-dimensional surface $E$
embedded in $\mathbb{R}^{d}$ and let $\Omega$ be the unit ball in
$\mathbb{R}^{d}$. Then%
\begin{equation}
\left\{  \int_{a}^{b}\int_{\mathbb{R}^{d}}\left\vert \mathcal{D}_{N}\left(
x+\tau\Omega\right)  \right\vert ^{2}dxd\tau\right\}  ^{1/2}\geqslant
c\,N^{1/2-1/(2k)}.\label{k-dim}%
\end{equation}

\end{corollary}

When $E$ is a $k$-dimensional sphere the above corollary gives an estimate of
the spherical cap discrepancy and we obtain the result in section 7.4 in
\cite{BC}. Observe that for $k=1$ the estimate (\ref{k-dim}) does not give a
divergent lower bound. Indeed, the example of $N$ equally spaced points on a
one dimensional circle in $\mathbb{R}^{d}$ shows that the dicrepancy can be
bounded. If $\Omega$ is a ball, $\mathcal{P}_{N}$ are $N$ equispaced points on
a one dimensional circle of length $1$, and $\mu$ is the natural measure on
$E$, then $\left(  x+\tau\Omega\right)  \cap E$ is an arc and $\left\vert
\mathcal{D}_{N}\left(  x+\tau\Omega\right)  \right\vert \leqslant1$. We
observe that Theorem \ref{Thm 1} shows that when $\alpha>1$ this phenomenon disappears.

\begin{example}
\label{example snowflake} The following iterative construction of a measure on
the snowflake curve is a classical example of an $\alpha$-dimensional measure
with $\alpha$ not an integer. Let $C_{0}$ be the boundary of an equilateral
triangle with a horizontal side and side length equal to $1$. At the stage $n$
of the construction $C_{n}$ contains $3\cdot4^{n}$ segments of length $3^{-n}$
and we construct $C_{n+1}$ replacing the middle third of every segment by the
other two sides of an "external" equilateral triangle. Let $\mu_{n}$ be the
probability measure that assignes the measure $\left(  3\cdot4^{n}\right)
^{-1}$ to every side in $C_{n}$. As $n$ goes to infinity these polygonal
curves $C_{n}$ approach the snowflake curve $C$ and the measure $\mu_{n}$
converges to a measure $\mu$. It is well known that $C$ has Hausdorff
dimension $\log_{3}4$ and that $\mu$ is, up to a normalization, the $\log
_{3}4$-dimensional Hausdorff measure restricted to $C$ and satisfies
(\ref{Ahlfors}) with $\alpha=\log_{3}4$. See \cite[Example 9.5]{Fal},
\cite{SS}, and Theorem 4.14 in \cite{Mattila1}. Hence, if $E=C$ and $\Omega$
is any convex body,%
\[
\left\{  \int_{a}^{b}\int_{SO\left(  d\right)  }\int_{\mathbb{R}^{d}%
}\left\vert \mathcal{D}_{N}\left(  x+\tau\sigma\Omega\right)  \right\vert
^{2}dxd\sigma d\tau\right\}  ^{1/2}\geqslant c\,N^{1/2-1/\left(  2\log
_{3}4\right)  }.
\]

\end{example}

\begin{example}
Let now $E$ be the bounded set such that $\partial E$ is the above snowflake
curve $C$, let $\mu$ be the $2$-dimensional Lebesgue measure and let
$\Omega=E$. We shall see in the Appendix that (\ref{Diff simm da sopra}) and
(\ref{Diff simm da sotto}) hold with $\beta=2-\log_{3}4$. Then Theorem
\ref{Thm 1} gives the estimate%
\[
\left\{  \int_{a}^{b}\int_{SO\left(  d\right)  }\int_{\mathbb{R}^{d}%
}\left\vert \mathcal{D}_{N}\left(  x+\tau\sigma\Omega\right)  \right\vert
^{2}dxd\sigma d\tau\right\}  ^{1/2}\geqslant c\,N^{\left(  \log_{3}2\right)
/2}.
\]

\end{example}

\subsection{Half-space discrepancy}

We now consider the half-space discrepancy. For $\rho\in\mathbb{R}$ and
$\Theta\in\Sigma_{d-1}$ (the unit $\left(  d-1\right)  $-dimensional sphere)
we consider the affine half-space%
\[
\Pi\left(  \rho,\Theta\right)  =\left\{  x\in\mathbb{R}^{d}:x\cdot\Theta
>\rho\right\}
\]
and, for every point distribution $\mathcal{P}_{N}=\left\{  z_{1},z_{2},\ldots
z_{N}\right\}  $, the associated discrepancy%
\begin{equation}
\mathcal{D}_{N}\left(  \Pi\left(  \rho,\Theta\right)  \right)  =\sum_{j=1}%
^{N}\chi_{\Pi\left(  \rho,\Theta\right)  }\left(  z_{j}\right)  -N\mu\left(
\Pi\left(  \rho,\Theta\right)  \right)  .\label{Discrep semispazi}%
\end{equation}
We have the following results.

\begin{theorem}
\label{Thm 3}Let $\mu$ be a positive Borel measure on $\mathbb{R}^{d}$ which
satisfies (\ref{Ahlfors}) for a given $0<\alpha\leqslant d$. Assume that the
support $E$ of $\mu$ is contained in a ball $B\left(  0,r_{0}\right)  $. Then
there exists $c>0$ such that for every point distribution $\mathcal{P}%
_{N}=\left\{  z_{1},z_{2},\ldots z_{N}\right\}  $ contained in $B\left(
0,r_{0}\right)  $ we have%
\[
\left\{  \int_{0}^{+\infty}\int_{\Sigma_{d-1}}\left\vert \mathcal{D}%
_{N}\left(  \Pi\left(  \rho,\Theta\right)  \right)  \right\vert ^{2}d\Theta
d\rho\right\}  ^{1/2}\geqslant c\,N^{1/2-1/\left(  2\alpha\right)  }.
\]

\end{theorem}

In the next section we shall see that the above result is sharp.

Since the discrepancy $\mathcal{D}_{N}\left(  \Pi\left(  \rho,\Theta\right)
\right)  $ has compact support as a function of $\rho$ and $\Theta$, one has
the following corollary.

\begin{corollary}
\label{cor5} There exists a constant $c>0$ such that for every point
distribution $\mathcal{P}_{N}=\left\{  z_{1},z_{2},\ldots z_{N}\right\}  $
there exist a half-space \ $\Pi^{\ast}=\Pi\left(  t^{\ast},\Theta^{\ast
}\right)  $ such that
\[
\left\vert \mathcal{D}_{N}\left(  \Pi^{\ast}\right)  \right\vert \geqslant
c\,N^{1/2-1/\left(  2\alpha\right)  }.
\]

\end{corollary}

The study of the half-space discrepancy goes back to a problem raised by K.
Roth for the unit disc and by P. Erd\H{o}s for the sphere. See \cite[Ch.
7.3]{BC}, \cite{chen},\cite[Ch. 3.2]{chenLN}, \cite{Erdos},\cite[p.124-125]%
{SchmidtNotes}.\medskip

If $E=\left[  0,1\right]  ^{d}$ or $E$ is the unit ball, and $\mu$ is the
Lebesgue measure, then a slightly weaker version Corollary \ref{cor5} has been
proved by J. Beck \cite{Beck0} using Fourier analysis. The sharp estimate in
Theorem \ref{Thm 3} was first proved by R. Alexander \cite{Alexander} using an
integral geometric approach (see also \cite{CMS}). We now show that it is
possible to obtain the sharp result using Fourier analysis. More precisely we
use Theorem \ref{Thm 1} to obtain the half-space discrepancy as a limit of the
discrepancies of balls with diverging radii.

If $E\subset\mathbb{R}^{d}$ is a manifold and $\mu$ is the Hausdorff measure
on $E$, then Theorem \ref{Thm 3} gives a lower bound for the discrepancy of
the \textit{spherical caps}, that are the intersections of $E$ and affine half-spaces.

If $E$ is an Euclidean sphere, then the spherical caps coincide with the
intersections of $E$ and balls. Then Corollary \ref{cor3} and Corollary
\ref{cor5} coincide and have been proved by J. Beck (see \cite[Theorem
24C]{BC}).

If $E$ is a compact set in $\mathbb{R}^{d}$ of Hausdorff dimension $\alpha$,
$\mu$ is the associated Hausdorff measure and (\ref{Ahlfors}) holds true, then
Corollary \ref{cor5} has been proved by H. Albrecher, J. Matousek and R. Tichy
(see \cite{AMT}) by adapting the technique of R. Alexander to the fractal
setting. See the remark at the end of page 244 in \cite{AMT}.

\section{Upper bounds}

The following theorem shows that Theorem \ref{Thm 1} is sharp when $\mu$ is
the Lebesgue measure, $\alpha=d$ and $\partial\Omega$ has finite $\left(
d-\beta\right)  $-dimensional Minkowski content (see Remark
\ref{Remark Minkowski}).

The discussion in Section 2.2 in \cite{AMT} shows that also the lower bound in
Theorem \ref{Thm 3} is \textquotedblleft best possible\textquotedblright.

\begin{theorem}
Let $E$ be a bounded Borel set in $\mathbb{R}^{d}$ with positive Lebesgue
measure and for every measurable set $F$ let
\[
\mu\left(  F\right)  =\frac{\left\vert E\cap F\right\vert }{\left\vert
E\right\vert }%
\]
be the Lebesgue measure restricted and normalized to $E$. Assume there exists
$c_{1}>0$ such that for every $0<r<\operatorname{diam}\left(  E\right)  $ and
every $x\in E$, we have%
\begin{equation}
c_{1}r^{d}\leqslant\mu\left(  B\left(  x,r\right)  \right)
\label{Alfhors regular}%
\end{equation}
(of course we have $\mu\left(  B\left(  x,r\right)  \right)  \leqslant
c_{2}r^{d}$). Let $\Omega\subseteq\mathbb{R}^{d}$ be a given bounded Borel set
that satisfies (\ref{Finite minkowski content}) for some $0<\beta\leqslant1$
and let $0<a<b$. Then there exists $c>0$ such that for every $N>0$ there
exists a distribution of $N$ points $\mathcal{P}_{N}=\left\{  z_{1}%
,z_{2},\ldots,z_{N}\right\}  \subset E$ such that%
\begin{equation}
\left\{  \int_{a}^{b}\int_{SO\left(  d\right)  }\int_{\mathbb{R}^{d}%
}\left\vert \mathcal{D}_{N}\left(  x+\tau\sigma\Omega\right)  \right\vert
^{2}dxd\sigma d\tau\right\}  ^{1/2}\leqslant cN^{1/2-\beta/\left(  2d\right)
}.\label{Discrep alto}%
\end{equation}

\end{theorem}

\begin{proof}
Since $\Omega$ and $E$ are bounded, then $\mathcal{D}_{N}\left(  x+\tau
\sigma\Omega\right)  $ has compact support, hence%
\begin{align*}
&  \int_{a}^{b}\int_{SO\left(  d\right)  }\int_{\mathbb{R}^{d}}\left\vert
\mathcal{D}_{N}\left(  x+\tau\sigma\Omega\right)  \right\vert ^{2}dxd\sigma
d\tau\\
&  =\int_{a}^{b}\int_{SO\left(  d\right)  }\int_{\left\{  \left\vert
x\right\vert \leqslant R\right\}  }\left\vert \mathcal{D}_{N}\left(
x+\tau\sigma\Omega\right)  \right\vert ^{2}dxd\sigma d\tau
\end{align*}
for a suitable $R>0$. Let us show that the above theorem follows applying
Corollary 8.2 in \cite{BCCGT} to the collection of sets
\[
\mathbb{G=}\left\{  \left(  x+\tau\sigma\Omega\right)  \cap E:\left\vert
x\right\vert \leqslant R,a\leqslant\tau\leqslant b,\sigma\in SO\left(
d\right)  \right\}  .
\]
First, observe that by Theorem 2 in \cite{GL} we can always decompose $E$ as
the union of $N$ sets of measure $N^{-1}$ and diameter of the order of
$N^{-1/d}$ as required by Corollary 8.2 in \cite{BCCGT}. It remains to check
that for every $\mathcal{G}\in\mathbb{G}$ we have%
\[
\left\vert \left\{  x\in\mathcal{G}:\operatorname*{dist}\left(  x,E\setminus
\mathcal{G}\right)  \leqslant t\right\}  \right\vert +\left\vert \left\{  x\in
E\setminus\mathcal{G}:\operatorname*{dist}\left(  x,\mathcal{G}\right)
\leqslant t\right\}  \right\vert \leqslant ct^{\beta}.
\]
Let $R=x+\tau\sigma\Omega$ and let $\mathcal{G}=R\cap E$. Then one has%
\begin{align*}
&  \left\{  x\in\mathcal{G}:\operatorname*{dist}\left(  x,E\setminus
\mathcal{G}\right)  \leqslant t\right\}  \cup\left\{  x\in E\setminus
\mathcal{G}:\operatorname*{dist}\left(  x,\mathcal{G}\right)  \leqslant
t\right\} \\
&  \subseteq\left\{  y\in E:\operatorname*{dist}\left(  y,\partial R\right)
\leqslant t\right\}  .
\end{align*}
Hence%
\begin{align*}
&  \left\vert \left\{  y\in\mathcal{G}:\operatorname*{dist}\left(
y,E\setminus\mathcal{G}\right)  \leqslant t\right\}  \right\vert +\left\vert
\left\{  y\in E\setminus\mathcal{G}:\operatorname*{dist}\left(  x,\mathcal{G}%
\right)  \leqslant t\right\}  \right\vert \\
&  \leqslant\left\vert \left\{  y\in E:\operatorname*{dist}\left(  y,\partial
R\right)  \leqslant t\right\}  \right\vert \\
&  \leqslant\left\vert \left\{  y\in\mathbb{R}^{d}:\operatorname*{dist}\left(
y,\partial\left(  \tau\sigma\Omega\right)  \right)  \leqslant t\right\}
\right\vert \\
&  =\left\vert \tau\left\{  w\in\mathbb{R}^{d}:\operatorname*{dist}\left(
w,\partial\Omega\right)  \leqslant\tau^{-1}t\right\}  \right\vert \leqslant
c\tau^{d-\beta}t^{\beta}\leqslant cb^{d-\beta}t^{\beta}.
\end{align*}

\end{proof}

\section{Proof of Theorem \ref{Thm 1}}

For the proof of Theorem \ref{Thm 1} we will use the Fourier transform of the
discrepancy function associated to the point distribution $\mathcal{P}%
_{N}=\left\{  z_{1},z_{2},\ldots,z_{N\ }\right\}  $,
\[
x\mapsto\mathcal{D}_{N}\left(  x,\tau,\sigma\right)  =\mathrm{card}\left(
\mathcal{P}_{N}\cap\left(  x+\tau\sigma\Omega\right)  \right)  -N\mu\left(
x+\tau\sigma\Omega\right)  .
\]

\begin{lemma}
\label{Lemma trasf discrep}We have%
\begin{align*}
\widehat{\mathcal{D}}_{N}\left(  \xi,\tau,\sigma\right)   &  =\int
_{\mathbb{R}^{d}}\mathcal{D}_{N}\left(  x,\tau,\sigma\right)  e^{-2\pi
i\xi\cdot x}dx\\
&  =\left\{  \sum_{j=1}^{N}e^{-2\pi i\xi\cdot z_{j}}-N\widehat{\mu}\left(
\xi\right)  \right\}  \overline{\widehat{\chi}_{\tau\sigma\Omega}\left(
\xi\right)  }.
\end{align*}

\end{lemma}

\begin{proof}
This is a simple consequence of the fact that%
\[
\mathcal{D}_{N}\left(  x,\tau,\sigma\right)  =\chi_{-\tau\sigma\Omega}%
\ast\left(  \sum_{j=1}^{N}\delta_{z_{j}}-N\mu\right)  \left(  x\right)  .
\]

\end{proof}

The following lemma is an extension of Lemma 4.2 in \cite{BCT}, see also
Theorem 8 in \cite{BGT}.

\begin{lemma}
\label{Lemma Leo}Let $\Omega$ be as in Theorem \ref{Thm 1}.\newline1) There
exist positive constants $c_{1},c_{2},\gamma$ and $\delta$ such that for
$\rho$ large enough%
\[
c_{1}\rho^{-\beta}\leqslant\int_{\left\{  \gamma\rho\leqslant\left\vert
\xi\right\vert \leqslant\delta\rho\right\}  }\left\vert \widehat{\chi_{\Omega
}}\left(  \xi\right)  \right\vert ^{2}d\xi\leqslant c_{2}\rho^{-\beta}.
\]
2) There exist positive constants $c_{3},c_{4},\gamma$ and $\delta$ such that
for $\left\vert \xi\right\vert $ large enough%
\[
c_{3}\,\left\vert \xi\right\vert ^{-d-\beta}\leqslant\int_{\gamma}^{\delta
}\int_{SO\left(  d\right)  }\left\vert \widehat{\chi}_{u\sigma\Omega}\left(
\xi\right)  \right\vert ^{2}dud\sigma\leqslant c_{4}\,\left\vert
\xi\right\vert ^{-d-\beta}.
\]

\end{lemma}

\begin{proof}
1) We first show that%
\[
\int_{\left\{  \left\vert \xi\right\vert \geqslant\rho\right\}  }\left\vert
\widehat{\chi_{\Omega}}\left(  \xi\right)  \right\vert ^{2}d\xi\leqslant
c\rho^{-\beta}.
\]
Let $\left\{  U_{j}\right\}  _{j=1}^{d}$ be a partition of the unit sphere
such that if $u=\left(  u_{1},\ldots,u_{d}\right)  \in U_{j}$ then $\left\vert
u_{j}\right\vert \geqslant\kappa>0$. Let $\left\{  e_{j}\right\}  _{j=1}^{d}$
be the canonical orthonormal basis in $\mathbb{R}^{d}$. Let $k\in\mathbb{N}$
and let $j=1,\ldots,d$ be given. Also let $h=\frac{1}{3\pi2^{k+1}}e_{j}$. Then
(\ref{Diff simm da sopra}) yields%
\begin{align*}
\kappa_{1}\left(  \frac{1}{3\pi2^{k+1}}\right)  ^{\beta} &  \geqslant
\left\vert \Omega\bigtriangleup\left(  \Omega+h\right)  \right\vert
=\int_{\mathbb{R}^{d}}\left\vert \chi_{\Omega}\left(  x\right)  -\chi
_{\Omega+h}\left(  x\right)  \right\vert ^{2}dx\\
&  =\int_{\mathbb{R}^{d}}\left\vert 1-e^{-2\pi ih\cdot\xi}\right\vert
^{2}\left\vert \widehat{\chi_{\Omega}}\left(  \xi\right)  \right\vert ^{2}%
d\xi\\
&  \geqslant\int_{\left\{  2^{k}\leqslant\left\vert \xi\right\vert
\leqslant2^{k+1}\right\}  }\left\vert 1-e^{-2\pi ih\cdot\xi}\right\vert
^{2}\left\vert \widehat{\chi_{\Omega}}\left(  \xi\right)  \right\vert ^{2}%
d\xi\\
&  \geqslant\int_{\left\{  2^{k}\leqslant\left\vert \xi\right\vert
\leqslant2^{k+1}:\frac{\xi}{\left\vert \xi\right\vert }\in U_{j}\right\}
}\left\vert 1-e^{-2\pi ih\cdot\xi}\right\vert ^{2}\left\vert \widehat
{\chi_{\Omega}}\left(  \xi\right)  \right\vert ^{2}d\xi.
\end{align*}
By our choice of $h$ and the definition of $U_{j}$, if $\frac{\xi}{\left\vert
\xi\right\vert }\in U_{j}$, then we have
\[
\frac{\kappa}{3}\leqslant\left\vert 2\pi h\cdot\xi\right\vert \leqslant
\frac{2}{3}%
\]
and therefore%
\[
\left\vert 1-e^{-2\pi ih\cdot\xi}\right\vert \geqslant c>0.
\]
Hence%
\[
\left(  \frac{1}{3\pi2^{k+1}}\right)  ^{\beta}\geqslant c\int_{\left\{
2^{k}\leqslant\left\vert \xi\right\vert \leqslant2^{k+1}:\frac{\xi}{\left\vert
\xi\right\vert }\in U_{j}\right\}  }\left\vert \widehat{\chi_{\Omega}}\left(
\xi\right)  \right\vert ^{2}d\xi.
\]
Repeating the above argument for every $j$ yields
\[
c2^{-k\beta}\geqslant\int_{\left\{  2^{k}\leqslant\left\vert \xi\right\vert
\leqslant2^{k+1}\right\}  }\left\vert \widehat{\chi_{\Omega}}\left(
\xi\right)  \right\vert ^{2}d\xi.
\]
Now let $k_{0}$ satisfy $2^{k_{0}}\leqslant\rho<2^{k_{0}+1}$. Then%
\begin{align*}
\int_{\left\{  \left\vert \xi\right\vert \geqslant\rho\right\}  }\left\vert
\widehat{\chi_{\Omega}}\left(  \xi\right)  \right\vert ^{2}d\xi &
\leqslant\sum_{k=k_{0}}^{+\infty}\int_{\left\{  2^{k}\leqslant\left\vert
\xi\right\vert \leqslant2^{k+1}\right\}  }\left\vert \widehat{\chi_{\Omega}%
}\left(  \xi\right)  \right\vert ^{2}d\xi\\
&  \leqslant c\sum_{k=k_{0}}^{+\infty}2^{-k\beta}\leqslant c\rho^{-\beta}.
\end{align*}
For $\gamma>0$ to be chosen later, we have%
\begin{align*}
&  \int_{\left\{  \left\vert \xi\right\vert \leqslant\gamma\rho\right\}
}\left\vert \xi\right\vert ^{2}\left\vert \widehat{\chi_{\Omega}}\left(
\xi\right)  \right\vert ^{2}d\xi\\
\leqslant &  \int_{\left\{  \left\vert \xi\right\vert \leqslant1\right\}
}\left\vert \widehat{\chi_{\Omega}}\left(  \xi\right)  \right\vert ^{2}%
d\xi+\sum_{1\leqslant2^{k}\leqslant\gamma\rho}\int_{\left\{  2^{k}%
\leqslant\left\vert \xi\right\vert \leqslant2^{k+1}\right\}  }\left\vert
\xi\right\vert ^{2}\left\vert \widehat{\chi_{\Omega}}\left(  \xi\right)
\right\vert ^{2}d\xi\\
\leqslant &  c+4\sum_{1\leqslant2^{k}\leqslant\gamma\rho}2^{2k}\int_{\left\{
2^{k}\leqslant\left\vert \xi\right\vert \leqslant2^{k+1}\right\}  }\left\vert
\widehat{\chi_{\Omega}}\left(  \xi\right)  \right\vert ^{2}d\xi\\
\leqslant &  c+c\sum_{2^{k}\leqslant\gamma\rho}2^{2k}2^{-k\beta}\\
= &  c+c\sum_{2^{k}\leqslant\gamma\rho}2^{\left(  2-\beta\right)  k}\leqslant
c\left(  \gamma\rho\right)  ^{2-\beta}.
\end{align*}
Let $n$ satisfy%
\[
t_{n+1}\leqslant\rho^{-1}\leqslant t_{n}%
\]
and let $h=t_{n}\overline{\Theta}$ with $\overline{\Theta}$ as in
(\ref{Diff simm da sotto}). Then%
\begin{align*}
\kappa_{2}\rho^{-\beta}\leqslant &  \kappa_{2}t_{n}^{\beta}\leqslant\left\vert
\Omega\bigtriangleup\left(  \Omega+h\right)  \right\vert =\int_{\mathbb{R}%
^{d}}\left\vert 1-e^{-2\pi ih\cdot\xi}\right\vert ^{2}\left\vert \widehat
{\chi_{\Omega}}\left(  \xi\right)  \right\vert ^{2}d\xi\\
\leqslant &  4\pi^{2}\int_{\left\{  \left\vert \xi\right\vert \leqslant
\gamma\rho\right\}  }\left\vert h\right\vert ^{2}\left\vert \xi\right\vert
^{2}\left\vert \widehat{\chi_{\Omega}}\left(  \xi\right)  \right\vert ^{2}%
d\xi+4\int_{\left\{  \gamma\rho<\left\vert \xi\right\vert <\delta\rho\right\}
}\left\vert \widehat{\chi_{\Omega}}\left(  \xi\right)  \right\vert ^{2}d\xi\\
&  +4\int_{\left\{  \left\vert \xi\right\vert \geqslant\delta\rho\right\}
}\left\vert \widehat{\chi_{\Omega}}\left(  \xi\right)  \right\vert ^{2}d\xi\\
\leqslant &  c\,t_{n}^{2}\,\left(  \gamma\rho\right)  ^{2-\beta}%
+4\int_{\left\{  \gamma\rho<\left\vert \xi\right\vert <\delta\rho\right\}
}\left\vert \widehat{\chi_{\Omega}}\left(  \xi\right)  \right\vert ^{2}%
d\xi+c\left(  \delta\rho\right)  ^{-\beta}%
\end{align*}
Observe that by (\ref{non lacunare}) we have%
\[
\rho t_{n}\leqslant\rho\kappa_{3}t_{n+1}\leqslant\kappa_{3}.
\]
If follows that%
\begin{align*}
4\int_{\left\{  \gamma\rho<\left\vert \xi\right\vert <\delta\rho\right\}
}\left\vert \widehat{\chi_{\Omega}}\left(  \xi\right)  \right\vert ^{2}d\xi &
\geqslant\kappa_{2}\rho^{-\beta}-c\left(  \gamma\rho\right)  ^{2-\beta}%
t_{n}^{2}-c\left(  \delta\rho\right)  ^{-\beta}\\
&  =\rho^{-\beta}\left(  \kappa_{2}-c\gamma^{2-\beta}\rho^{2}t_{n}^{2}%
-c\delta^{-\beta}\right)  \\
&  \geqslant c\rho^{-\beta}%
\end{align*}
for $\gamma$ small enough and $\delta$ large enough.

2) Let $\gamma$ and $\delta$ be as in 1). If $d\Theta$ is the normalized
surface measure on the sphere $\Sigma_{d-1}$, then%
\begin{align*}
&  \int_{\gamma}^{\delta}\int_{SO\left(  d\right)  }\left\vert \widehat{\chi
}_{u\sigma\Omega}\left(  \xi\right)  \right\vert ^{2}dud\sigma\\
&  =\int_{\gamma}^{\delta}\int_{SO\left(  d\right)  }u^{2d}\left\vert
\widehat{\chi}_{\Omega}\left(  u\sigma^{-1}\xi\right)  \right\vert
^{2}dud\sigma\\
&  =\int_{\gamma}^{\delta}u^{2d}\int_{\Sigma_{d-1}}\left\vert \widehat{\chi
}_{\Omega}\left(  u\left\vert \xi\right\vert \Theta\right)  \right\vert
^{2}dud\Theta\\
&  =\left\vert \xi\right\vert ^{-1-2d}\int_{\gamma\left\vert \xi\right\vert
}^{\delta\left\vert \xi\right\vert }\int_{\Sigma_{d-1}}t^{2d}\left\vert
\widehat{\chi}_{\Omega}\left(  t\Theta\right)  \right\vert ^{2}dtd\Theta\\
&  =\left\vert \xi\right\vert ^{-d}\int_{\gamma\left\vert \xi\right\vert
}^{\delta\left\vert \xi\right\vert }\int_{\Sigma_{d-1}}\left(  t\left\vert
\xi\right\vert ^{-1}\right)  ^{d+1}\left\vert \widehat{\chi}_{\Omega}\left(
t\Theta\right)  \right\vert ^{2}t^{d-1}dtd\Theta\\
&  \approx\left\vert \xi\right\vert ^{-d}\int_{\left\{  \delta\xi
\leqslant\left\vert \zeta\right\vert \leqslant\gamma\left\vert \xi\right\vert
\right\}  }\left\vert \widehat{\chi}_{\Omega}\left(  \zeta\right)  \right\vert
^{2}d\zeta\approx\left\vert \xi\right\vert ^{-d-\beta}.
\end{align*}

\end{proof}

In the proof of Theorem \ref{Thm 1} we need a positive function with positive
and compactly supported Fourier transform. Let $\psi$ be a positive smooth
radial function with support in $\left\{  x\in\mathbb{R}^{d}:\left\vert
x\right\vert \leqslant\frac{1}{2}\right\}  $ such that $\left\Vert
\psi\right\Vert _{1}=1$, and let%
\[
K\left(  x\right)  =\left(  \widehat{\psi}\left(  x\right)  \right)
^{2}\text{.}%
\]
Then
\[
0\leqslant\psi\ast\psi\left(  \xi\right)  =\widehat{K}\left(  \xi\right)
\leqslant1,
\]
and $\widehat{K}\left(  \xi\right)  =0$ for $\left\vert \xi\right\vert
\geqslant1$. For every $M\geqslant1$ let%
\begin{equation}
K_{M}\left(  x\right)  =M^{d}K\left(  Mx\right)  .\label{def KM}%
\end{equation}
Since $\widehat{K}_{M}\left(  \xi\right)  =\widehat{K}\left(  \xi/M\right)  $
we have%
\[
0\leqslant\widehat{K}_{M}\left(  \xi\right)  \leqslant1,
\]
and $\widehat{K}_{M}\left(  \xi\right)  =0$ for $\left\vert \xi\right\vert
\geqslant M$. Moreover, for every given $L>0$ there exists $c>0$ such that%
\[
0\leqslant K\left(  x\right)  \leqslant c\left\{
\begin{array}
[c]{lc}%
1 & \text{if }\left\vert x\right\vert \leqslant1\\
\left\vert x\right\vert ^{-L} & \text{if }\left\vert x\right\vert \geqslant1
\end{array}
\right.
\]
and therefore%
\begin{equation}
\left\vert K_{M}\left(  x\right)  \right\vert \leqslant\left\{
\begin{array}
[c]{lc}%
M^{d} & \text{if }\left\vert x\right\vert \leqslant\frac{1}{M},\\
M^{d-L}x^{-L} & \text{if }\left\vert x\right\vert \geqslant\frac{1}{M}.
\end{array}
\right.  \label{Stima KM}%
\end{equation}

\begin{lemma}
\label{Lemma alfa palla}Let $\mu$ be as in Theorem \ref{Thm 1}. There exists
$c>0$ such that for every $M\geq1$ and for every $z\in\mathbb{R}^{d}$ we have%
\[
\left\vert K_{M}\ast\mu\left(  z\right)  \right\vert \leqslant cM^{d-\alpha}.
\]

\end{lemma}

\begin{proof}
Let $z\in\mathbb{R}^{d}$ and let $B_{2^{-k}}=\left\{  x\in\mathbb{R}%
^{d}:\left\vert z-x\right\vert \leqslant2^{-k}\right\}  $. Then%
\begin{align*}
&  K_{M}\ast\mu\left(  z\right)  =\int_{\mathbb{R}^{d}}K_{M}\left(
z-x\right)  d\mu\left(  x\right) \\
&  =\int_{\left\{  \left\vert z-x\right\vert \geqslant1\right\}  }K_{M}\left(
z-x\right)  d\mu\left(  x\right)  +\sum_{k=0}^{+\infty}\int_{B_{2^{-k}%
}\backslash B_{2^{-k-1}}}K_{M}\left(  z-x\right)  d\mu\left(  x\right) \\
&  =I_{1}\left(  z\right)  +I_{2}\left(  z\right)  .
\end{align*}

Using (\ref{Stima KM}) with $L=\alpha$, we readily obtain%
\[
I_{1}\left(  z\right)  \leqslant c\int_{\left\{  \left\vert z-x\right\vert
\geqslant1\right\}  }M^{d-\alpha}d\mu\left(  x\right)  \leqslant cM^{d-\alpha
}.
\]
Moreover, using (\ref{Stima KM}) with $L$ large enough, we have%
\begin{align*}
I_{2}\left(  z\right)   &  =\sum_{2^{-k}M\leqslant1}\int_{B_{2^{-k}}\backslash
B_{2^{-k-1}}}K_{M}\left(  z-x\right)  d\mu\left(  x\right) \\
&  +\sum_{2^{-k}M>1}\int_{B_{2^{-k}}\backslash B_{2^{-k-1}}}K_{M}\left(
z-x\right)  d\mu\left(  x\right) \\
&  \leqslant c\sum_{2^{-k}M\leqslant1}\int_{B_{2^{-k}}\backslash B_{2^{-k-1}}%
}M^{d}d\mu\left(  x\right) \\
&  +c\sum_{2^{-k}M>1}\int_{B_{2^{-k}}\backslash B_{2^{-k-1}}}M^{d-L}\left\vert
z-x\right\vert ^{-L}d\mu\left(  x\right) \\
&  \leqslant c\sum_{2^{-k}M\leqslant1}M^{d}\mu\left(  B_{2^{-k}}\right)
+c\sum_{2^{-k}M>1}M^{d-L}2^{kL}\mu\left(  B_{2^{-k}}\right) \\
&  \leqslant cM^{d}\sum_{2^{-k}M\leqslant1}2^{-k\alpha}+cM^{d-L}\sum
_{2^{-k}M>1}2^{k\left(  L-\alpha\right)  }\\
&  \leqslant cM^{d}M^{-\alpha}+cM^{d-L}M^{L-\alpha}=cM^{d-\alpha}.
\end{align*}

\end{proof}

The next lemma is a generalization of an elegant result of Cassels and
Mongomery (see \cite[Chapter 6]{montgomery}).

\begin{lemma}
\label{Lemma Cassels}Let $\mu$ be as in Theorem \ref{Thm 1}. There exist
$c_{1},c_{2}>0$ such that for every $N$, for every point distribution
$\mathcal{P}_{N}=\left\{  z_{1},z_{2},\ldots,z_{N}\right\}  $, and for every
$M\geqslant c_{1}N^{1/\alpha}$ we have%
\[
\int_{\left\{  1\leqslant\left\vert \xi\right\vert \leqslant M\right\}
}\left\vert \sum_{j=1}^{N}e^{-2\pi i\xi\cdot z_{j}}-N\widehat{\mu}\left(
\xi\right)  \right\vert ^{2}d\xi\geqslant c_{2}NM^{d}.
\]

\end{lemma}

\begin{proof}
Let $M\geqslant1$. We have%
\begin{align}
&  \int_{\left\{  1\leqslant\left\vert \xi\right\vert \leqslant M\right\}
}\left\vert \sum_{j=1}^{N}e^{-2\pi i\xi\cdot z_{j}}-N\widehat{\mu}\left(
\xi\right)  \right\vert ^{2}d\xi\nonumber\\
= &  \int_{\left\{  \left\vert \xi\right\vert \leqslant M\right\}  }\left\vert
\sum_{j=1}^{N}e^{-2\pi i\xi\cdot z_{j}}-N\widehat{\mu}\left(  \xi\right)
\right\vert ^{2}d\xi\nonumber\\
&  -\int_{\left\{  \left\vert \xi\right\vert \leqslant1\right\}  }\left\vert
\sum_{j=1}^{N}e^{-2\pi i\xi\cdot z_{j}}-N\widehat{\mu}\left(  \xi\right)
\right\vert ^{2}d\xi\nonumber\\
\geqslant &  \int_{\left\vert \xi\right\vert \leqslant M}\left\vert \sum
_{j=1}^{N}e^{-2\pi i\xi\cdot z_{j}}-N\widehat{\mu}\left(  \xi\right)
\right\vert ^{2}d\xi-cN^{2}.\nonumber
\end{align}
Let $K_{M}$ be as in (\ref{def KM}), then%
\begin{align*}
&  \int_{\left\{  \left\vert \xi\right\vert \leqslant M\right\}  }\left\vert
\sum_{j=1}^{N}e^{-2\pi i\xi\cdot z_{j}}-N\widehat{\mu}\left(  \xi\right)
\right\vert ^{2}d\xi\\
&  \geqslant\int_{\mathbb{R}^{d}}\widehat{K}_{M}\left(  \xi\right)  \left\vert
\sum_{j=1}^{N}\left(  e^{-2\pi i\xi\cdot z_{j}}-\widehat{\mu}\left(
\xi\right)  \right)  \right\vert ^{2}d\xi\\
&  =\int_{\mathbb{R}^{d}}\widehat{K}_{M}\left(  \xi\right)  \sum_{j=1}%
^{N}\left(  e^{2\pi i\xi\cdot z_{j}}-\overline{\widehat{\mu}\left(
\xi\right)  }\right)  \sum_{k=1}^{N}\left(  e^{-2\pi i\xi\cdot z_{k}}%
-\widehat{\mu}\left(  \xi\right)  \right)  d\xi\\
&  =\int_{\mathbb{R}^{d}}\widehat{K}_{M}\left(  \xi\right)  \\
&  \times\sum_{k,j=1}^{N}\left(  e^{2\pi i\xi\cdot\left(  z_{j}-z_{k}\right)
}-e^{2\pi i\xi\cdot z_{j}}\widehat{\mu}\left(  \xi\right)  -\overline
{\widehat{\mu}\left(  \xi\right)  }e^{-2\pi i\xi\cdot z_{k}}+\left\vert
\widehat{\mu}\left(  \xi\right)  \right\vert ^{2}\right)  d\xi\\
&  =\sum_{k,j=1}^{N}\int_{\mathbb{R}^{d}}\widehat{K}_{M}\left(  \xi\right)
e^{2\pi i\xi\cdot\left(  z_{j}-z_{k}\right)  }d\xi-2N\sum_{j=1}^{N}%
\int_{\mathbb{R}^{d}}\widehat{K}_{M}\left(  \xi\right)  \widehat{\mu}\left(
\xi\right)  e^{2\pi i\xi\cdot z_{j}}d\xi\\
&  +N^{2}\int_{\mathbb{R}^{d}}\widehat{K}_{M}\left(  \xi\right)  \left\vert
\widehat{\mu}\left(  \xi\right)  \right\vert ^{2}d\xi\\
&  =\sum_{k,j=1}^{N}K_{M}\left(  z_{j}-z_{k}\right)  -2N\sum_{j=1}^{N}%
K_{M}\ast\mu\left(  z_{j}\right)  +N^{2}\int_{\mathbb{R}^{d}}\widehat{K}%
_{M}\left(  \xi\right)  \left\vert \widehat{\mu}\left(  \xi\right)
\right\vert ^{2}d\xi.
\end{align*}
Since the terms in the double sum $1\leqslant j,k\leqslant N$ are positive,
the double sum is bounded below by $NK_{M}\left(  0\right)  $. Moreover also
the last integral is positive. Hence%
\[
\int_{\left\{  \left\vert \xi\right\vert \leqslant M\right\}  }\left\vert
\sum_{j=1}^{N}e^{-2\pi i\xi\cdot z_{j}}-N\widehat{\mu}\left(  \xi\right)
\right\vert ^{2}d\xi\geqslant NK_{M}\left(  0\right)  -2N\sum_{j=1}^{N}%
K_{M}\ast\mu\left(  z_{j}\right)  .
\]
Since $K_{M}\left(  0\right)  =M^{d}$ and since, by Lemma
\ref{Lemma alfa palla} we have $\left\vert K_{M}\ast\mu\left(  z_{j}\right)
\right\vert \leqslant cM^{d-\alpha}$, we obtain%
\[
\int_{\left\{  \left\vert \xi\right\vert \leqslant M\right\}  }\left\vert
\sum_{j=1}^{N}e^{-2\pi i\xi\cdot z_{j}}-N\widehat{\mu}\left(  \xi\right)
\right\vert ^{2}d\xi\geqslant NM^{d}-cN^{2}M^{d-\alpha}.
\]
We recall that $\alpha\leqslant d$. Then, if $M\geqslant c_{1}N^{1/\alpha}$
with $c_{1}$ large enough, we have%
\begin{align*}
&  \int_{\left\{  1\leqslant\left\vert \xi\right\vert \leqslant M\right\}
}\left\vert \sum_{j=1}^{N}e^{-2\pi i\xi\cdot z_{j}}-N\widehat{\mu}\left(
\xi\right)  \right\vert ^{2}d\xi\geqslant NM^{d}-cN^{2}M^{d-\alpha}-cN^{2}\\
&  =NM^{d}\left(  1-cNM^{-\alpha}-cNM^{-d}\right)  \geqslant c_{2}NM^{d}.
\end{align*}

\end{proof}

\begin{proof}
[Proof of Theorem \ref{Thm 1}]By Plancherel identity we have%
\[
\int_{\mathbb{R}^{d}}\left\vert \mathcal{D}_{N}\left(  x,\tau,\sigma\right)
\right\vert ^{2}dx=\int_{\mathbb{R}^{d}}\left\vert \widehat{\mathcal{D}}%
_{N}\left(  \xi,\tau,\sigma\right)  \right\vert ^{2}d\xi.
\]
Let $\gamma$ and $\delta$ be as in Lemma \ref{Lemma Leo}. Then, by Lemma
\ref{Lemma Leo} point 2), for every $M\geqslant1$ we have%
\begin{align*}
&  \int_{\gamma}^{\delta}\int_{SO\left(  d\right)  }\int_{\mathbb{R}^{d}%
}\left\vert \widehat{\mathcal{D}}_{N}\left(  \xi,\tau,\sigma\right)
\right\vert ^{2}d\xi d\sigma d\tau\\
= &  \int_{\mathbb{R}^{d}}\left\vert \sum_{j=1}^{N}e^{-2\pi i\xi\cdot z_{j}%
}-N\widehat{\mu}\left(  \xi\right)  \right\vert ^{2}\int_{\gamma}^{\delta}%
\int_{SO\left(  d\right)  }\left\vert \widehat{\chi}_{\tau\sigma\Omega}\left(
\xi\right)  \right\vert ^{2}d\sigma d\tau d\xi\\
\geqslant &  \int_{\left\{  1\leqslant\left\vert \xi\right\vert \leqslant
M\right\}  }\left\vert \sum_{j=1}^{N}e^{-2\pi i\xi\cdot z_{j}}-N\widehat{\mu
}\left(  \xi\right)  \right\vert ^{2}\int_{\gamma}^{\delta}\int_{SO\left(
d\right)  }\left\vert \widehat{\chi}_{\tau\sigma\Omega}\left(  \xi\right)
\right\vert ^{2}d\sigma d\tau d\xi\\
\geqslant &  \int_{\left\{  1\leqslant\left\vert \xi\right\vert \leqslant
M\right\}  }\left\vert \sum_{j=1}^{N}e^{-2\pi i\xi\cdot z_{j}}-N\widehat{\mu
}\left(  \xi\right)  \right\vert ^{2}d\xi\\
&  \times\left\{  \inf_{1\leqslant\left\vert \xi\right\vert \leqslant M}%
\int_{\gamma}^{\delta}\int_{SO\left(  d\right)  }\left\vert \widehat{\chi
}_{\tau\sigma\Omega}\left(  \xi\right)  \right\vert ^{2}d\sigma d\tau\right\}
\\
&  \geqslant cM^{-d-\beta}\int_{\left\{  1\leqslant\left\vert \xi\right\vert
\leqslant M\right\}  }\left\vert \sum_{j=1}^{N}e^{-2\pi i\xi\cdot z_{j}%
}-N\widehat{\mu}\left(  \xi\right)  \right\vert ^{2}d\xi.
\end{align*}
Hence, by Lemma \ref{Lemma Cassels}, if $M=c_{1}N^{1/\alpha}$, we have%
\begin{align*}
&  \int_{a}^{b}\int_{SO\left(  d\right)  }\int_{\mathbb{R}^{d}}\left\vert
\mathcal{D}_{N}\left(  x,\tau,\sigma\right)  \right\vert ^{2}dxd\sigma d\tau\\
&  =\int_{a}^{b}\int_{SO\left(  d\right)  }\int_{\mathbb{R}^{d}}\left\vert
\widehat{\mathcal{D}}_{N}\left(  \xi,\tau,\sigma\right)  \right\vert ^{2}d\xi
d\sigma d\tau\\
&  \geqslant cM^{-d-\beta}\int_{\left\{  1\leqslant\left\vert \xi\right\vert
\leqslant M\right\}  }\left\vert \sum_{j=1}^{N}e^{-2\pi i\xi\cdot z_{j}%
}-N\widehat{\mu}\left(  \xi\right)  \right\vert ^{2}d\xi\\
&  \geqslant cM^{-d-\beta}NM^{d}=cN^{1-\beta/\alpha}.
\end{align*}

\end{proof}

\section{Proof of Theorem \ref{Thm 3}}

The characteristic function of a half-space can be obtained as a limit of
characteristic functions of balls of diverging radii. Hence we start with a
lemma on the Fourier transform on the characteristic functions $\chi_{rB}$ of
the balls $rB=\left\{  x\in\mathbb{R}^{d}:\left\vert x\right\vert <r\right\}
$.

\begin{lemma}
\label{Da sotto Ball}There exist $c_{1},c_{2}>0$ such that for $R\left\vert
\xi\right\vert \geqslant c_{1}$ we have%
\[
\frac{1}{R}\int_{R}^{2R}\left\vert \widehat{\chi}_{rB}\left(  \xi\right)
\right\vert ^{2}dr\geqslant c_{2}R^{d-1}\left\vert \xi\right\vert ^{-d-1}.
\]

\end{lemma}

\begin{proof}
The Fourier transform of $\chi_{rB}\left(  x\right)  $ can be expressed in
terms of a Bessel function,%
\[
\widehat{\chi}_{rB}\left(  \xi\right)  =r^{d}\widehat{\chi}_{B}\left(
r\xi\right)  =r^{d/2}\left\vert \xi\right\vert ^{-d/2}J_{d/2}\left(  2\pi
r\left\vert \xi\right\vert \right)  .
\]
Bessel functions have the asymptotic expansion%
\[
J_{d/2}\left(  2\pi u\right)  =\pi^{-1}u^{-1/2}\cos\left(  2\pi u-\left(
d+1\right)  \frac{\pi}{4}\right)  +E_{d}\left(  u\right)
\]
with $\left\vert E_{d}\left(  u\right)  \right\vert \leqslant c_{d}\left\vert
u\right\vert ^{-3/2}$ (see e.g. Lemma 3.11 in SW). Then, if $R\left\vert
\xi\right\vert \geqslant c_{1}$ with $c_{1}$ sufficiently large,%
\begin{align*}
&  \frac{1}{R}\int_{R}^{2R}\left\vert \widehat{\chi}_{rB}\left(  \xi\right)
\right\vert ^{2}dr=\frac{1}{R}\int_{R}^{2R}r^{d}\left\vert \xi\right\vert
^{-d}\left\vert J_{d/2}\left(  2\pi r\left\vert \xi\right\vert \right)
\right\vert ^{2}dr\\
= &  \left\vert \xi\right\vert ^{-2d}\frac{1}{R\left\vert \xi\right\vert }%
\int_{R\left\vert \xi\right\vert }^{2R\left\vert \xi\right\vert }%
u^{d}\left\vert J_{d/2}\left(  2\pi u\right)  \right\vert ^{2}du\\
= &  \left\vert \xi\right\vert ^{-2d}\frac{1}{R\left\vert \xi\right\vert }%
\int_{R\left\vert \xi\right\vert }^{2R\left\vert \xi\right\vert }%
u^{d-1}\left\vert \pi^{-1}\cos\left(  2\pi u-\left(  d+1\right)  \frac{\pi}%
{4}\right)  +u^{1/2}E_{d}\left(  u\right)  \right\vert ^{2}du\\
\geqslant &  \left\vert \xi\right\vert ^{-2d}\left(  R\left\vert
\xi\right\vert \right)  ^{d-1}\\
&  \times\frac{1}{R\left\vert \xi\right\vert }\int_{R\left\vert \xi\right\vert
}^{2R\left\vert \xi\right\vert }\left\vert \pi^{-1}\cos\left(  2\pi u-\left(
d+1\right)  \frac{\pi}{4}\right)  +u^{1/2}E_{d}\left(  u\right)  \right\vert
^{2}du\\
\geqslant &  \,c\,R^{d-1}\left\vert \xi\right\vert ^{-d-1}.
\end{align*}

\end{proof}

In the next lemma we estimate the discrepancy associated to the family of
balls $x+rB$, where $B$ is the unit ball centered at the origin. With a small
change of the previous notation, for any $x\in\mathbb{R}^{d}$ and $r>0$ we set%
\[
\mathcal{D}_{N}\left(  x,r\right)  =\mathrm{card}\left(  \mathcal{P}_{N}%
\cap\left(  x+rB\right)  \right)  -N\mu\left(  x+rB\right)  .
\]

\begin{lemma}
Under the assumption of Theorem \ref{Thm 3} there exists $c>0$ such that for
$R$ large enough%
\[
\frac{1}{R}\int_{R}^{2R}\int_{\mathbb{R}^{d}}\left\vert \mathcal{D}_{N}\left(
x,r\right)  \right\vert ^{2}dxdr\geqslant cR^{d-1}N^{1-1/\alpha}.
\]

\end{lemma}

\begin{proof}
By Lemma \ref{Lemma trasf discrep} we have%
\[
\widehat{\mathcal{D}}_{N}\left(  \xi,r\right)  =\left\{  \sum_{j=1}%
^{N}e^{-2\pi i\xi\cdot z_{j}}-N\widehat{\mu}\left(  \xi\right)  \right\}
\overline{\widehat{\chi}_{rB}\left(  \xi\right)  }.
\]
Then, by the Plancherel identity and Lemma \ref{Lemma Cassels} with
$M=cN^{1/\alpha}$, we have%
\begin{align*}
&  \frac{1}{R}\int_{R}^{2R}\int_{\mathbb{R}^{d}}\left\vert \mathcal{D}%
_{N}\left(  x,r\right)  \right\vert ^{2}dxdr=\frac{1}{R}\int_{R}^{2R}%
\int_{\mathbb{R}^{d}}\left\vert \widehat{\mathcal{D}}_{N}\left(  \xi,r\right)
\right\vert ^{2}d\xi dr\\
&  =\frac{1}{R}\int_{R}^{2R}\int_{\mathbb{R}^{d}}\left\vert \sum_{j=1}%
^{N}e^{-2\pi i\xi\cdot z_{j}}-\widehat{\mu}\left(  \xi\right)  \right\vert
^{2}\left\vert \widehat{\chi}_{rB}\left(  \xi\right)  \right\vert ^{2}d\xi
dr\\
&  \geqslant\int_{\left\{  1\leqslant\left\vert \xi\right\vert \leqslant
M\right\}  }\left\vert \sum_{j=1}^{N}e^{-2\pi i\xi\cdot z_{j}}-N\widehat{\mu
}\left(  \xi\right)  \right\vert ^{2}\left\{  \frac{1}{R}\int_{R}%
^{2R}\left\vert \widehat{\chi}_{rB}\left(  \xi\right)  \right\vert
^{2}dr\right\}  d\xi\\
&  \geqslant cR^{d-1}M^{-d-1}\int_{\left\{  1\leqslant\left\vert
\xi\right\vert \leqslant M\right\}  }\left\vert \sum_{j=1}^{N}e^{2\pi
i\xi\cdot z_{j}}-N\overline{\widehat{\mu}\left(  \xi\right)  }\right\vert
^{2}d\xi\\
&  \geqslant cR^{d-1}M^{-d-1}NM^{d}\geqslant cR^{d-1}N^{1-1/\alpha}.
\end{align*}

\end{proof}

\begin{proof}
[Proof of Theorem \ref{Thm 3}]Let $r>r_{0}$. Since both the set $E$ and the
points $\mathcal{P}_{N}$ are contained in the ball $B\left(  0,r_{0}\right)
$, if $\left\vert x\right\vert <r-r_{0}$ then $B\left(  0,r_{0}\right)
\subset B\left(  x,r\right)  $ and one can easily check that%
\[
\mathcal{D}_{N}\left(  x,r\right)  =0.
\]
Similarly, if $\left\vert x\right\vert >r+r_{0}$, then $B\left(
0,r_{0}\right)  \cap B\left(  x,r\right)  =\varnothing$ and also in this case
we have $\mathcal{D}_{N}\left(  x,r\right)  =0$. Then the previous lemma
gives, for $R>r_{0}$,%
\begin{align*}
& c_{1}\,R^{d-1}N^{1-1/\alpha}   \leqslant\frac{1}{R}\int_{R}^{2R}%
\int_{\mathbb{R}^{d}}\left\vert \mathcal{D}_{N}\left(  x,r\right)  \right\vert
^{2}dxdr\\
  = & \frac{1}{R}\int_{R}^{2R}\int_{\left\{  r-r_{0}\leqslant\left\vert
x\right\vert \leqslant r+r_{0}\right\}  }\left\vert \mathcal{D}_{N}\left(
x,r\right)  \right\vert ^{2}dxdr\\
  = & \frac{1}{R}\int_{R}^{2R}\int_{r-r_{0}}^{r+r_{0}}\int_{\Sigma_{d-1}%
}\left\vert \mathcal{D}_{N}\left(  \rho\Theta,r\right)  \right\vert ^{2}%
\rho^{d-1}d\rho d\Theta dr\\
  \leqslant & \frac{1}{R}\left(  2R+r_{0}\right)  ^{d-1}\int_{R}^{2R}%
\int_{r-r_{0}}^{r+r_{0}}\int_{\Sigma_{d-1}}\left\vert \mathcal{D}_{N}\left(
\rho\Theta,r\right)  \right\vert ^{2}d\rho d\Theta dr\\
  = & \frac{1}{R}\left(  2R+r_{0}\right)  ^{d-1}\int_{R}^{2R}\int_{-r_{0}%
}^{r_{0}}\int_{\Sigma_{d-1}}\left\vert \mathcal{D}_{N}\left(  \left(
\rho+r\right)  \Theta,r\right)  \right\vert ^{2}d\rho d\Theta dr.
\end{align*}

Then, there exists $c>0$ such that for every $R$ large enough,
\begin{align*}
c\,N^{1-1/\alpha}  &  \leqslant\int_{-r_{0}}^{r_{0}}\int_{\Sigma_{d-1}}\left[
\frac{1}{R}\int_{R}^{2R}\left\vert \mathcal{D}_{N}\left(  \left(
\rho+r\right)  \Theta,r\right)  \right\vert ^{2}dr\right]  d\Theta d\rho\\
&  =\int_{-r_{0}}^{r_{0}}\int_{\Sigma_{d-1}}\left[  \int_{1}^{2}\left\vert
\mathcal{D}_{N}\left(  \left(  \rho+\omega R\right)  \Theta,\omega R\right)
\right\vert ^{2}d\omega\right]  d\Theta d\rho.
\end{align*}
We claim that as $R\rightarrow+\infty$ the discrepancy associated to these
balls converges to the discrepancy associated to half-spaces, that is we claim
that%
\begin{align*}
&  \lim_{R\rightarrow+\infty}\int_{-r_{0}}^{r_{0}}\int_{\Sigma_{d-1}}\left[
\int_{1}^{2}\left\vert \mathcal{D}_{N}\left(  \left(  \rho+\omega R\right)
\Theta,\omega R\right)  \right\vert ^{2}d\omega\right]  d\Theta d\rho\\
&  =\int_{-r_{0}}^{r_{0}}\int_{\Sigma_{d-1}}\left\vert \mathcal{D}_{N}\left(
\Pi\left(  \Theta,\rho\right)  \right)  \right\vert ^{2}d\Theta d\rho.
\end{align*}

Observe that%
\begin{align*}
B\left(  \left(  \rho+\omega R\right)  \Theta,\omega R\right)   &  =\left\{
x:\left\vert x-\rho\Theta-\omega R\Theta\right\vert ^{2}<\omega^{2}%
R^{2}\right\} \\
&  =\left\{  x:\left\vert x-\rho\Theta\right\vert ^{2}<2\omega R\left(
x-\rho\Theta\right)  \cdot\Theta\right\} \\
&  =\left\{  x:\frac{\left\vert x-\rho\Theta\right\vert ^{2}}{2\omega R}%
+\rho<x\cdot\Theta\right\}  .
\end{align*}
Hence, for every $x\in\mathbb{R}^{d}$,%
\[
\lim_{R\rightarrow+\infty}\chi_{B\left(  \left(  \rho+\omega R\right)
\Theta,\omega R\right)  }\left(  x\right)  =\chi_{\Pi\left(  \Theta
,\rho\right)  }\left(  x\right)  .
\]
Then, for every $\Theta$, $\rho$ and $\omega$, we have
\begin{align*}
&  \lim_{R\rightarrow+\infty}\mathcal{D}_{N}\left(  \left(  \rho+\omega
R\right)  \Theta,\omega R\right) \\
&  =\lim_{R\rightarrow+\infty}\left[  \sum_{j=1}^{N}\chi_{B\left(  \left(
\rho+\omega R\right)  \Theta,\omega R\right)  }\left(  z_{j}\right)
-N\mu\left(  B\left(  \left(  \rho+\omega R\right)  \Theta,\omega R\right)
\right)  \right] \\
&  =\sum_{j=1}^{N}\chi_{\Pi\left(  \Theta,\rho\right)  }\left(  z_{j}\right)
-N\mu\left(  \Pi\left(  \Theta,\rho\right)  \right)  =\mathcal{D}_{N}\left(
\Pi\left(  \Theta,\rho\right)  \right)  .
\end{align*}

By the dominate convergence theorem we have%
\begin{align*}
&  \lim_{R\rightarrow+\infty}\int_{-r_{0}}^{r_{0}}\int_{\Sigma_{d-1}}\left[
\int_{1}^{2}\left\vert \mathcal{D}_{N}\left(  \left(  \rho+\omega R\right)
\Theta,\omega R\right)  \right\vert ^{2}d\omega\right]  d\Theta d\rho\\
&  =\lim_{R\rightarrow+\infty}\int_{-r_{0}}^{r_{0}}\int_{\Sigma_{d-1}%
}\left\vert \mathcal{D}_{N}\left(  \Pi\left(  \Theta,\rho\right)  \right)
\right\vert ^{2}d\Theta d\rho.
\end{align*}

\end{proof}

\section{Appendix}

\begin{figure}[ptb]
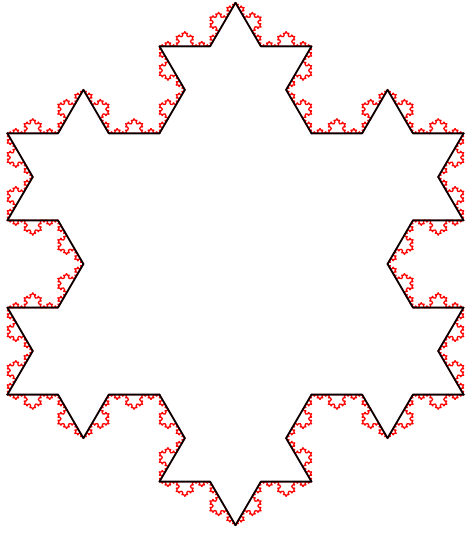


\caption{The snowflake}%
\label{fig:snowslake}%
\end{figure}Let $C$ be the snowflake curve constructed in Example
\ref{example snowflake}. The next theorem is a particular case of Theorems
\ref{Thm 1} and \ref{Thm 3} when both the support $E$ of the measure $\mu$ and
the set $\Omega$ used to define the discrepancy coincide with the interior of
the snowflake.

\begin{theorem}
Let $E=\Omega$ be the open bounded set whose boundary is the snowflake curve
$C$ and let $\mu$ be the Borel measure on $\mathbb{R}^{d}$ defined by%
\[
\mu\left(  F\right)  =\frac{\left\vert F\cap E\right\vert }{\left\vert
E\right\vert }\text{.}%
\]
i) There exist positive constants $a,b$ and $c$ such that for every
distribution $\mathcal{P}_{N}=\left\{  z_{1},z_{2},\ldots z_{N}\right\}  $ of
$N$ points we have%
\[
\left\{  \int_{a}^{b}\int_{SO\left(  d\right)  }\int_{\mathbb{R}^{d}%
}\left\vert \mathcal{D}_{N}\left(  x+\tau\sigma\Omega\right)  \right\vert
^{2}dxd\sigma d\tau\right\}  ^{1/2}\geqslant c\,N^{\frac{\log_{3}4}{4}}.
\]
ii) For every $0<a<b$ there exists $c>0$ such that for every $N$ there exists
a distribution $\mathcal{P}_{N}=\left\{  z_{1},z_{2},\ldots z_{N}\right\}
\subseteq E$ of $N$ points such that%
\[
\left\{  \int_{a}^{b}\int_{SO\left(  d\right)  }\int_{\mathbb{R}^{d}%
}\left\vert \mathcal{D}_{N}\left(  x+\tau\sigma\Omega\right)  \right\vert
^{2}dxd\sigma d\tau\right\}  ^{1/2}\leqslant c\,N^{\frac{\log_{3}4}{4}}.
\]

\end{theorem}

\begin{proof}
i) Let us show that we can apply Theorem \ref{Thm 1} to this setting with
$\alpha=2$ and $\beta=2-\log_{3}4$. Clearly (\ref{Ahlfors}) holds true with
$\alpha=2$. We only have to prove (\ref{Diff simm da sopra}) and
(\ref{Diff simm da sotto}).

Let us show that $\Omega$ has finite $\left(  2-\log4/\log3\right)
$-dimensional Minkowski content. Given $0<t<1$ let $n$ be such that
$3^{-n-1}\leqslant t<3^{-n}$. The construction of $C$ shows that for every
$y\in C$ there exists $z\in C_{n}$ such that $\operatorname*{dist}\left(
y,z\right)  \leqslant3^{-n}$. Then%
\[
\left\{  x\in\mathbb{R}^{d}:d\left(  x,C\right)  \leqslant t\right\}
\subset\left\{  x\in\mathbb{R}^{d}:d\left(  x,C_{n}\right)  \leqslant
2\cdot3^{-n}\right\}  .
\]
Since $C_{n}$ has length $3\cdot\left(  \frac{4}{3}\right)  ^{n}$,
\begin{align*}
\left\vert \left\{  x\in\mathbb{R}^{d}:d\left(  x,C\right)  \leqslant
t\right\}  \right\vert  &  \leqslant\left\vert \left\{  x\in\mathbb{R}%
^{d}:d\left(  x,C_{n}\right)  \leqslant2\cdot3^{-n}\right\}  \right\vert \\
&  \leqslant3\cdot\left(  \frac{4}{3}\right)  ^{n}\cdot\left(  2\cdot
3^{-n}\right)  \leqslant6t^{2-\frac{\log4}{\log3}}.
\end{align*}

By Remark \ref{Remark Minkowski} $\Omega$ satisfies (\ref{Diff simm da sopra})
with $\beta=\frac{\log4}{\log3}$.

It remains to show that there exists a decreasing sequence $t_{n}$ satisfying
(\ref{non lacunare}) and a direction $\overline{\Theta}$ such that
\[
\kappa_{2}\,t_{n}^{2-\log_{3}4.}\leqslant\left\vert \left(  t_{n}%
\overline{\Theta}+\Omega\right)  \bigtriangleup\Omega\right\vert .
\]
Let $\overline{\Theta}=\left(  0,1\right)  $, let $t_{n}=\sqrt{3}/2\cdot
3^{-n}$ and let $h=t_{n}\left(  0,1\right)  $. We stop the construction of the
snowflake curve $C$ at the step $n$ we can write%
\[
\Omega=K_{n}\cup F_{n}%
\]
where $K_{n}$ is an open bounded set whose boundary is $C_{n}$ and $F_{n}$ is
a bounded set which is the disjoint union of triangles of side lengths
$3^{-k-1}$ for $k\geqslant n$. More precisely for every $k\geqslant n$,
$F_{n}$ contains $3\cdot4^{k}$ disjoint triangles of side lengths $3^{-k-1}$.
The measure of $F_{n}$ is therefore given by%
\[
\left\vert F_{n}\right\vert =\sum_{k=n}^{+\infty}3\cdot4^{k}\frac{1}%
{2}3^{-k-1}\frac{\sqrt{3}}{2}3^{-k-1}=\frac{3\sqrt{3}}{20}\left(  \frac{4}%
{9}\right)  ^{n}.
\]
To estimate from below the size of $\left\vert \left(  \Omega+h\right)
\bigtriangleup\Omega\right\vert $ observe that
\[
\left(  \Omega+h\right)  \bigtriangleup\Omega\supseteq\left(  K_{n}%
\bigtriangleup\left(  K_{n}+h\right)  \right)  \setminus\left(  F_{n}%
\cup\left(  F_{n}+h\right)  \right)  .
\]
The part of $K_{n}\bigtriangleup\left(  K_{n}+h\right)  $ originating from the
horizontal segments contains $4^{n}$ disjoint trapezoids and rectangles of
area $3^{-n}t_{n}=\sqrt{3}/2\cdot9^{-n}$ (see figure \ref{fig:sim diff}).
Hence%
\[
\left\vert K_{n}\bigtriangleup\left(  K_{n}+h\right)  \right\vert
\geqslant\frac{\sqrt{3}}{2}\left(  \frac{4}{9}\right)  ^{n}.
\]
Since $\left\vert F_{n}\right\vert =\left\vert F_{n}+h\right\vert
=\frac{3\sqrt{3}}{20}\left(  \frac{4}{9}\right)  ^{n}$, then we have%
\begin{align*}
\left\vert \Omega\bigtriangleup\left(  \Omega+h\right)  \right\vert  &
\geqslant\left\vert K_{n}\bigtriangleup\left(  K_{n}+h\right)  \right\vert
-\left\vert F_{n}\right\vert -\left\vert F_{n}+h\right\vert \\
&  \geqslant\frac{\sqrt{3}}{2}\left(  \frac{4}{9}\right)  ^{n}-\frac{3\sqrt
{3}}{10}\left(  \frac{4}{9}\right)  ^{n}=\frac{\sqrt{3}}{5}\left(  \frac{4}%
{9}\right)  ^{n}\\
&  =\frac{\sqrt{3}}{5}\left(  2/\sqrt{3}\right)  ^{\log_{3}4-2}t_{n}%
^{2-\log_{3}4}.
\end{align*}

\end{proof}

\begin{figure}[ptb]
\begin{tikzpicture}[line cap=round,line join=round,>=triangle 45,x=2cm,y=2cm]
\tikzset{l1/.style={line width=0.5pt,color=red}}
\draw[line width=0.5pt,color=black] (0,0)-- (0.333,0)-- (0.5,0.289)-- (0.667,0)-- (1,0);
\draw[line width=0.5pt,color=black] (1,0)-- (1.167,0.289)-- (1,0.577)-- (1.333,0.577)-- (1.5,0.866);
\draw[line width=0.5pt,color=black] (1.5,0.866)-- (1.667,0.577)-- (2,0.577)-- (1.833,0.289)-- (2,0);
\draw[line width=0.5pt,color=black] (2,0)-- (2.333,0)-- (2.5,0.289)-- (2.667,0)-- (3,0);
\draw[line width=0.5pt,color=black] (1.5,-2.598)-- (1.333,-2.309)-- (1,-2.309)-- (1.167,-2.021)-- (1,-1.732);
\draw[line width=0.5pt,color=black] (1,-1.732)-- (0.667,-1.732)-- (0.5,-2.021)-- (0.333,-1.732)-- (0,-1.732);
\draw[line width=0.5pt,color=black] (0,-1.732)-- (0.167,-1.443)-- (0,-1.155)-- (0.333,-1.155)-- (0.5,-0.866);
\draw[line width=0.5pt,color=black] (0.5,-0.866) -- (0.333,-0.577) -- (0,-0.577)-- (0.167,-0.289)-- (0,0);
\draw[line width=0.5pt,color=black] (3,0)-- (2.833,-0.289)-- (3,-0.577)-- (2.667,-0.577)-- (2.5,-0.866);
\draw[line width=0.5pt,color=black] (2.5,-0.866)-- (2.667,-1.155)-- (3,-1.155)-- (2.833,-1.443)-- (3,-1.732);
\draw[line width=0.5pt,color=black] (3,-1.732)-- (2.667,-1.732)-- (2.5,-2.021)-- (2.333,-1.732)-- (2,-1.732);
\draw[line width=0.5pt,color=black] (2,-1.732)-- (1.833,-2.021)-- (2,-2.309)-- (1.667,-2.309)-- (1.5,-2.598);
\draw[l1] (0,-0.13)-- (0.333,-0.13)-- (0.5,0.155)-- (0.667,-0.13)-- (1,-0.13);
\draw[l1] (1,-0.13)-- (1.167,0.159)-- (1,0.447)-- (1.333,0.447)-- (1.5,0.736);
\draw[l1] (1.5,0.736)-- (1.667,0.447)-- (2,0.447)-- (1.833,0.159)-- (2,-0.13);
\draw[l1] (2,-0.13)-- (2.333,-0.13)-- (2.5,0.159)-- (2.667,-0.13)-- (3,-0.13);
\draw[l1] (1.5,-2.728)-- (1.333,-2.439)-- (1,-2.439)-- (1.167,-2.151)-- (1,-1.862);
\draw[l1] (1,-1.862)-- (0.667,-1.862)-- (0.5,-2.151)-- (0.333,-1.862)-- (0,-1.862);
\draw[l1] (0,-1.862)-- (0.167,-1.573)-- (0,-1.285)-- (0.333,-1.285)-- (0.5,-0.996);
\draw[l1] (0.5,-0.996)-- (0.333,-0.707)-- (0,-0.707)-- (0.167,-0.419)-- (0,-0.13);
\draw[l1] (3,-0.13)-- (2.833,-0.419)-- (3,-0.707)-- (2.667,-0.707)-- (2.5,-0.996);
\draw[l1] (2.5,-0.996)-- (2.667,-1.285)-- (3,-1.285)-- (2.833,-1.573)-- (3,-1.862);
\draw[l1] (3,-1.862)-- (2.667,-1.862)-- (2.5,-2.151)-- (2.333,-1.862)-- (2,-1.862);
\draw[l1] (2,-1.862)-- (1.833,-2.151)-- (2,-2.439)-- (1.667,-2.439)-- (1.5,-2.728);
\draw[line width=0.5pt,color=gray,fill=gray,fill opacity=0.1] (1,0.577) -- (1.075,0.447) -- (1.333,0.447) -- (1.333,0.577) -- cycle;
\draw[line width=0.2pt,color=gray,fill=gray,fill opacity=0.1] (1.666,0.447) -- (1.666,0.577) -- (2,0.577) -- (1.925,0.447) -- cycle;
\draw[line width=0.2pt,color=gray,fill=gray,fill opacity=0.1] (0.666,0) -- (0.666,-0.13) -- (1,-0.13) -- (1,0) -- cycle;
\draw[line width=0.2pt,color=gray,fill=gray,fill opacity=0.1] (0,0) -- (0.075,-0.13) -- (0.333,-0.13) -- (0.333,0) -- cycle;
\draw[line width=0.2pt,color=gray,fill=gray,fill opacity=0.1] (2,0) -- (2,-0.13) -- (2.333,-0.13) -- (2.333,0) -- cycle;
\draw[line width=0.2pt,color=gray,fill=gray,fill opacity=0.1] (2.666,0) -- (2.666,-0.13) -- (2.924,-0.13) -- (3,0) -- cycle;
\draw[line width=0.2pt,color=gray,fill=gray,fill opacity=0.1] (1,-2.439) -- (1.075,-2.309) -- (1.333,-2.309) -- (1.333,-2.439) -- cycle;
\draw[line width=0.2pt,color=gray,fill=gray,fill opacity=0.1] (1.666,-2.309) -- (1.666,-2.44) -- (2,-2.439) -- (1.924,-2.309) -- cycle;
\draw[line width=0.2pt,color=gray,fill=gray,fill opacity=0.1] (0.666,-1.862) -- (0.666,-1.732) -- (1,-1.732) -- (1,-1.862) -- cycle;
\draw[line width=0.2pt,color=gray,fill=gray,fill opacity=0.1] (0,-1.862) -- (0.075,-1.732) -- (0.333,-1.732) -- (0.333,-1.862) -- cycle;
\draw[line width=0.2pt,color=gray,fill=gray,fill opacity=0.1] (2,-1.862) -- (2,-1.732) -- (2.333,-1.732) -- (2.333,-1.862) -- cycle;
\draw[line width=0.2pt,color=gray,fill=gray,fill opacity=0.1] (2.666,-1.862) -- (2.666,-1.732) -- (2.924,-1.732) -- (3,-1.862) -- cycle;
\draw[line width=0.2pt,color=gray,fill=gray,fill opacity=0.1] (0,-0.707) -- (0.333,-0.707) -- (0.333,-0.577) -- (0.075,-0.577) -- cycle;
\draw[line width=0.2pt,color=gray,fill=gray,fill opacity=0.1] (2.667,-0.577) -- (2.667,-0.707) -- (3,-0.707) -- (2.925,-0.577) -- cycle;
\draw[line width=0.2pt,color=gray,fill=gray,fill opacity=0.1] (0,-1.155) -- (0.333,-1.155) -- (0.333,-1.285) -- (0.075,-1.285) -- cycle;
\draw[line width=0.2pt,color=gray,fill=gray,fill opacity=0.1] (2.667,-1.285) -- (2.667,-1.155) -- (3,-1.155) -- (2.925,-1.285) -- cycle;
\end{tikzpicture}
\caption{The symmetric difference}%
\label{fig:sim diff}%
\end{figure}
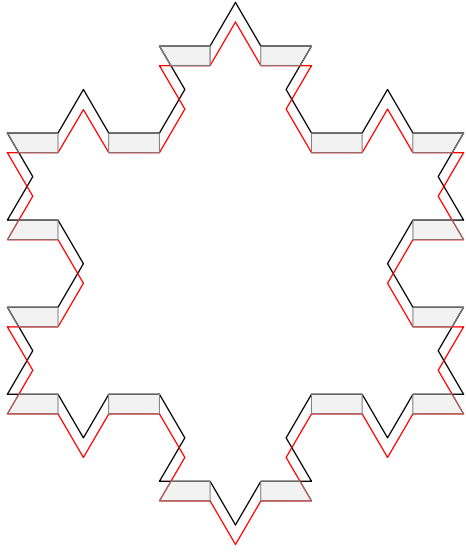

We now show that for every $0<\beta<1$ there exist sets that satisfy
(\ref{Diff simm da sopra}) and (\ref{Diff simm da sotto}).

\begin{example}
\label{n gamma}Let $0<\beta<1$, let $\gamma=\frac{\beta}{1-\beta}$, and for
every positive integer $n$ let%
\[
z_{n}=\frac{1}{n^{\gamma}}-\frac{1}{\left(  n+1\right)  ^{\gamma}},
\]
and%
\[
R_{n}=\left[  \frac{1}{n^{\gamma}}-\frac{1}{3}z_{n},\frac{1}{n^{\gamma}%
}\right]  \times\left[  0,1\right]  .
\]
Finally let%
\[
\Omega=\bigcup_{n=1}^{+\infty}R_{n}%
\]
be the union of the above rectangles. If $h=\left(  h_{1},0\right)  $ with
$0<\left\vert h_{1}\right\vert \leqslant\frac{1}{2}$ and $n_{0}$ is defined by%
\[
\frac{1}{3}z_{n_{0}+1}<\left\vert h_{1}\right\vert \leqslant\frac{1}%
{3}z_{n_{0}},
\]
then, for $n=1,\ldots,n_{0}$, we have $\left(  R_{n+1}+h\right)  \cap
R_{n}=\varnothing$, and since $\left\vert h\right\vert \approx z_{n_{0}%
}\approx\frac{1}{n_{0}^{\gamma+1}}$ we have%
\[
\left\vert \Omega\bigtriangleup\left(  \Omega+h\right)  \right\vert
\geqslant\sum_{n=1}^{n_{0}}\left\vert R_{n}\bigtriangleup\left(
R_{n}+h\right)  \right\vert =\sum_{n=1}^{n_{0}}\left\vert h_{1}\right\vert
=n_{0}\left\vert h_{1}\right\vert \approx c\left\vert h\right\vert
^{1-\frac{1}{\gamma+1}}=c\left\vert h\right\vert ^{\beta}.
\]
Now let $h=\left(  h_{1},h_{2}\right)  $. Since%
\[
A\bigtriangleup B\subset\left(  A\bigtriangleup C\right)  \cup\left(
B\bigtriangleup C\right)  ,
\]
we have%
\begin{align*}
&  \Omega\bigtriangleup\left(  h+\Omega\right)  \\
&  \subset\left[  \left(  \left(  h_{1},0\right)  +\left(  0,h_{2}\right)
+\Omega\right)  \bigtriangleup\left(  \left(  h_{1},0\right)  +\Omega\right)
\right]  \cup\left[  \left(  \left(  h_{1},0\right)  +\Omega\right)
\bigtriangleup\Omega\right]
\end{align*}
and therefore%
\[
\left\vert \Omega\bigtriangleup\left(  h+\Omega\right)  \right\vert
\leqslant\left\vert \left(  \left(  0,h_{2}\right)  +\Omega\right)
\bigtriangleup\Omega\right\vert +\left\vert \left(  \left(  h_{1},0\right)
+\Omega\right)  \bigtriangleup\Omega\right\vert .
\]
Then%
\[
\left\vert \Omega\bigtriangleup\left(  \Omega+\left(  h_{1},0\right)  \right)
\right\vert \leqslant\sum_{n=1}^{n_{0}}\left\vert R_{n}\bigtriangleup\left(
R_{n}+\left(  h_{1},0\right)  \right)  \right\vert +\frac{1}{n_{0}^{\gamma}%
}\approx c\left\vert h_{1}\right\vert ^{\beta}.
\]
Also, if $\left\vert h_{2}\right\vert \leqslant\frac{1}{2}$, then%
\[
\left\vert \Omega\bigtriangleup\left(  \Omega+\left(  0,h_{2}\right)  \right)
\right\vert =2\left\vert \Omega\right\vert \left\vert h_{2}\right\vert
\]
and thus%
\[
\left\vert \Omega\bigtriangleup\left(  h+\Omega\right)  \right\vert \leqslant
c\left\vert h\right\vert ^{\beta}.
\]

\end{example}


\begin{thebibliography}{99}                                                                                               %


\bibitem {AMT}H. Albrecher, J. Matou\v{s}ek, R. Tichy, \textit{Discrepancy of
point sequences on fractal sets}, Publ. Math. Debrecen \textbf{56} (2000), 233--249.

\bibitem {Alexander}R. Alexander, \textit{Principles of a new method in the
study of irregularities of distribution}. Invent. Math. \textbf{103} (1991), 279--296.

\bibitem {Beck0}J. Beck, \textit{On a problem of K. F. Roth concerning
irregularities of point distribution}. Invent. Math. \textbf{74} (1983), 477--487.

\bibitem {beck}J. Beck, \textit{Irregularities of point distribution I}. Acta
Math. \textbf{159} (1988), 1--49.

\bibitem {BC}J. Beck and W.W.L. Chen, \textquotedblleft Irregularities of
distribution\textquotedblright. Cambridge University Press (1987).

\bibitem {BCCGT}L. Brandolini, W.W.L. Chen, L. Colzani, G. Gigante, G.
Travaglini, \textit{Discrepancy and numerical integration on metric measure
spaces}. J. Geom. Anal. \textbf{29} (2019), 328--369.

\bibitem {BCT}L. Brandolini, L. Colzani, G. Travaglini, \textit{Average decay
of Fourier transforms and integer points in polyhedra}. Ark. Mat. \textbf{35}
(1997), 253--275

\bibitem {BBG}L. Brandolini, B. Gariboldi, G. Gigante, \textit{On a sharp
lemma of Cassels and Montgomery on manifolds}. Math. Ann. \textbf{379} (2021), 1807--1834.

\bibitem {BGT}L. Brandolini, G. Gigante, G. Travaglini, \textit{Irregularities
of distribution and average decay of Fourier transforms}. In: A panorama of
discrepancy theory, 159--220, Lecture Notes in Math., 2107, Springer, Cham, 2014.

\bibitem {BT22}L. Brandolini, G. Travaglini, \textit{Irregularities of
distribution and geometry of planar convex sets}. Adv. Math. \textbf{396}
(2022), Paper No. 108162.

\bibitem {chazelle}B. Chazelle. \textquotedblleft The discrepancy method.
Randomness and complexity\textquotedblright. Cambridge University Press (2000).

\bibitem {CMS}B.Chazelle, J. Matousek, M. Sharir, \textit{An elementary
approach to lower bounds in geometric discrepancy}. Discrete Comput. Geom.
\textbf{13} (1995), 363--381.

\bibitem {chen}W.W.L. Chen, Results and problems old and new in discrepancy
theory, in \textquotedblleft Discrepancy Theory\textquotedblright\ (D. Bilyk,
J. Dick, F. Pillichshammer Editors). De Gruyter (2020), 21--42.

\bibitem {chenLN}W.W.L. Chen, \textquotedblleft Lectures on Discrepancy
Theory\textquotedblright. Unpublished, https://www.williamchen-mathematics.info/ln.html.

\bibitem {CT}L. L. Cristea, R.T. Tichy, \textit{Discrepancies of point
sequences on the Sierpinski carpet}. Math. Slovaca, \textbf{53 }(2003), 351--367.

\bibitem {DT}M. Drmota, R.F. Tichy, \textquotedblleft Sequences, discrepancies
and applications\textquotedblright. Springer (1997).

\bibitem {Erdos}P. Erd\H{o}s, \textit{Problems and results on diophantine
approximations}. Compositio Math. \textbf{16} (1964), 52--65.

\bibitem {Fal2}K. Falconer, "Techniques in fractal geometry", John Wiley \&
Sons (1997).

\bibitem {Fal}K. Falconer, \textquotedblleft Fractal
Geometry\textquotedblright, John Wiley \& Sons (2014).

\bibitem {GL}G. Gigante, P. Leopardi, \textit{Diameter bounded equal measure
partitions of Ahlfors regular metric measure spaces}. Discrete Comput. Geom.
\textbf{57} (2017), 419--430.

\bibitem {Lap-Pea}M. L. Lapidus, E. P. J. Pearse, \textit{A tube formula for
the Koch snowflake curve, with applications to complex dimensions}. J. London
Math. Soc. \textbf{74} (2006), 397--414.

\bibitem {matousek}J. Matousek, \textquotedblleft Geometric discrepancy. An
illustrated guide. Springer (2010).

\bibitem {Mattila1}P. Mattila, Geometry of sets and measures in Euclidean
spaces, Cambridge University Press (1995).

\bibitem {Mattila}P. Mattila, Fourier analysis and Hausdorff dimension,
Cambridge University Press (2015).

\bibitem {montgomery}H. Montgomery, \textquotedblleft Ten Lectures on the
interface between Analytic Number Theory and Harmonic
Analysis\textit{\textquotedblright}. American Mathematical Society (1994).

\bibitem {roth}K. Roth. \textit{On irregularities of distribution}.
Mathematika \textbf{1} (1954), 73--79.

\bibitem {schmidt}W. Schmidt, \textit{Irregularities of distribution IV}.
Invent. Math. \textbf{7} (1969), 55--82.

\bibitem {SchmidtNotes}Schmidt W.M, \textquotedblleft Lectures on
irregularities of distribution\textquotedblright. Tata Institute of
Fundamental Research, Bombay (1977).

\bibitem {SS}E.M. Stein, R. Shakarchi, Real Analysis, Princeton University
Press, 2005

\bibitem {travaglini}G. Travaglini, \textquotedblleft Number theory, Fourier
analysis and Geometric Discrepancy\textquotedblright. Cambridge University
Press (2014).
\end{thebibliography}
\end{document}